\documentclass{gen-j-l}

\newtheorem*{teo*}{Theorem}
\newtheorem{teo}{Theorem}[section]
\newtheorem{lema}[teo]{Lemma}
\newtheorem{prop}[teo]{Proposition}
\newtheorem{cor}[teo]{Corollary}
\newtheorem*{cor*}{Corollary}
\newtheorem*{conjec*}{Conjecture}

\theoremstyle{definition}
\newtheorem{definition}[teo]{Definition}
\newtheorem{example}[teo]{Example}

\theoremstyle{remark}
\newtheorem{remark}[teo]{Remark}

\numberwithin{equation}{section}

\usepackage[all,knot,arc]{xy}
\usepackage{tikz-cd} % canonic algebra
\usepackage{paralist} % enumerate, itemize, inparaenum

\newcommand{\md}{\operatorname{mod}}

\newcommand{\add}{\operatorname{add}}
\newcommand{\Coh}{\operatorname{Coh}}

\newcommand{\Hom}{\operatorname{Hom}}
\newcommand{\Ext}{\operatorname{Ext}}
\newcommand{\End}{\operatorname{End}}

\newcommand{\gl}{\operatorname{gldim}}
\newcommand{\glf}{\operatorname{sgldim}}
\newcommand{\dpr}{\operatorname{pd}}
\newcommand{\din}{\operatorname{id}}

\begin{document}

\title{Piecewise Hereditary Incidence Algebras}
\thanks{The first named author has been supported by the tematic project of Fapesp 2014/09310-5. The second named author acknowledges support from CAPES, in the form of a PhD Scholarship, PhD made at programa de Matem\'atica, IME-USP, (Brazil).}


%    author one information
\author{Eduardo do N. Marcos}
\address{}
\curraddr{Dto de Matem\'atica, Instituto de Matem\'atica e Estat\'istica, Universidade S\~ao Paulo, Rua do Mat\~ao 1010, Cidade Universit\'aria, CEP 05508-090, S\~ao Paulo-SP, Brasil}
\email{enmarcos@ime.usp.br}
\thanks{}

%    author two information
\author{Marcelo Moreira}
\address{}
\curraddr{Dto de Matem\'atica, 
Instituto de Ci\^encias Exatas, Universidade Federal de Alfenas, Campus Sede, 
Rua Gabriel Monteiro da Silva 700, Centro, CEP 37130-001, Alfenas-MG, Brasil}
\email{marcelo.moreira@unifal-mg.edu.br}
\thanks{}

%    \subjclass is required by all journals except JAG.

\date{}

\begin{abstract}
Let $K\Delta$ be the incidence algebra associated with a finite poset $(\Delta,\preceq)$ over the algebraically closed field $K$. We present a study of incidence algebras $K\Delta$ that are piecewise hereditary, which we denominate PHI algebras. We investigate the strong global dimension, the simply conectedeness and the one-point extension algebras over a PHI algebras. 

 We also give a positive answer to the so-called Skowro\'nski problem for $K\Delta$ a PHI algebra  which is not of wild quiver type. That is for this kind of algebra we 
 show that  $HH^1(K\Delta)$ is trivial if, and only if, $K\Delta$ is a simply connected algebra. 
We determine an upper bound for the strong global dimension of PHI algebras; furthermore, we extend this result to sincere algebras proving that the strong global dimension of a sincere piecewise hereditary algebra is less or equal than three.

\end{abstract}

\maketitle

%\keywords{
%incidence algebra, piecewise hereditary algebra, simply connected, strong global dimension, one-point extension algebra.\\

\section{Introduction}

Throughout the paper, $K$ denotes an algebraically closed field. All algebra will be  finite dimensional basic associative $K$-algebra. Using Gabriel's theorem we will assume all algebras to be of the form $KQ/I$, where $Q$ is a finite quiver and $I$ is an admissible ideal. All modules will be finite dimensional right module.

An algebra $A=KQ/I$ can be considered as a small $K$-category where of objects are the vertices of the quiver, given two vertices $v$, $w$ the set of homomorphism from $v$ to $w$ is 
the $K$-vector space $vAw$, composition is multiplication in $A$. So we can talk about subcategories, etc...

Incidence algebras were introduced in the mid-1960s as a natural way of studying some combinatorial problems. In the representation theory of finite dimensional algebras, the incidence algebras have been the subject of many investigations (see, for instance, \cite{lad1}, \cite{red1}, \cite{ass-pla-red-tre} and \cite{ass-cas-mar-tre}). We will focus the study on 
incidence algebras $K\Delta$ associated with a finite poset $\Delta$ over $K$.
We remark that $K\Delta$ is the isomorphic to the algebra $KQ/I$, where the quiver $Q$ is the Hasse diagram and $I$ is the ideal generated by all commuting relations, i.e. the difference of any pair of parallel paths are in $I$.

Our purpose is to study incidence algebras which are piecewise hereditary, we call them PHI algebras, \emph{piecewise hereditary 
incidence algebras}. We investigate the \textit{strong global dimension}, the \emph{simply conectedeness} and the \emph{one-point extension algebras} of PHI algebras. 

The paper is organized as follows. 
Section \ref{sec2} is devoted to fixing the notation and briefly recalling the necessary concepts and results about 
incidence algebras and piecewise hereditary algebras. Section \ref{sec3} is dedicated to study the simply conectedeness of 
PHI algebras in order to solve the so called Skowro\'nski problem:
\begin{quotation}
Is $A$ simply connected if and only if $HH^1(A)=0$?
\end{quotation}

Let $\mathcal{A}$ and $\mathcal{B}$ be an abelian categories. In this paper the notation $\mathcal{A} \cong' \mathcal{B}$ means that $\mathcal{A}$ is derived equivalently to $\mathcal{B}$, that is $D^b(\mathcal{A})$ and $D^b(\mathcal{B})$ are equivalent as triangulated categories, we also use the notation 
$D^b (A)$ for the category
$D^b(\md A)$, where $\md A$ denotes the category of finitely generated modules over a finite dimensional algebra $A$.

For PHI algebras we show that  the answer for Skowro\'nsky's question is positive if the PHI algebra is not of
the  quiver type for a wild quiver. We conjecture that the restrictive hypothesis on the former statements is not necessary. 

\begin{teo*}
Let $K\Delta$ be a PHI algebra that is not of wild-quiver type. Then $HH^1(\!K\Delta\!)\!$ $=0$, if and only if 
the algebra $K\Delta$ is simply connected.
\end{teo*}

Section \ref{sec4} is dedicated to the study of the global dimension of PHI algebras. In particular, we show that the representation-finite PHI algebras have global dimension less or equal to two.

Skowro\'nski, Happel and Zacharia have introduced a new homological notion called the strong global dimension. Let $A$ be an algebra. The strong global dimension of $A$, $\glf A$, is defined to be the maximum of the width of indecomposable, minimal complexes in $C^b(\mathcal{P}_A)$. We use an alternative definition of the strong global dimension \cite{alv-lem-mar}. In \cite{ker-sko-yam-zac}, O. Kerner, A. Skowro\'{n}ski, K. Yamagata and D. Zacharia proved that the strong global dimension of a­ finite dimensional radical square zero algebra $A$ over an algebraically closed field is finite if and only if $A$ is piecewise hereditary. Later in \cite{hap-zac} Happel and Zacharia generalized the result showing that an algebras has finite strong global dimension if and only if it is piecewise hereditary. In section \ref{sec5} we determine an upper bound of the strong global dimension for sincere algebras piecewise hereditary algebras.
\begin{teo*}
The strong global dimension of any sincere and piecewise hereditary algebra $A$ is at most three.
\end{teo*}

We apply this result for PHI algebras and get the following corollary.

\begin{cor*} 
The strong global dimension of any PHI algebra is less or equal than $3$.
\end{cor*}

Usually the PHI algebras are not of global dimension two, so this gives a large class of examples of algebras which have global dimension equal strong global dimension. In general it is hard to find classes of algebras where these two invariants are equal.

Let $p_1,\dots,p_n$ be a set of positive integers and let $\mathbb{X} = \mathbb{X}(p_1,\dotsc,p_n)$ be a weighted projective line of type $p_1,\dotsc,p_n$, in the sense of \cite{gei-len}. Let $\Coh \mathbb{X}$ be the category of coherent sheaves on $\mathbb{X}(p_1,\dotsc,p_n)$. 
For the PHI algebras $K\Delta \cong' \Coh \mathbb{X}$, in Section \ref{Phiasfeixe} we study the canonical sincere $K\Delta$-module $M$ and the one-point extension algebra $K\Delta[M]$. 
Let $A$ a representation-infinite quasi-tilted of domestic-sheaf type. We show that the canonical sincere $A$-module $M$ is exceptional. We conjecture that this module is always excepcional in the case of PHI algebras. This condition is necessary to create new PHI algebras of wild type as one-point extension algebra $K\Delta[M]$. 
%We conclude that if $A$ is a representation-infinite quasi-tilted algebra of domestic sheaf type with a sincere indecomposable module $M$, then $M$ is exceptional.

\section{Preliminaries} \label{sec2}
In this Section, for the sake of completeness, we will recall some definitions. The reader should see the references for more detail.

We begin with the definition of incidence algebras. There are several equivalent ways of defining incidence algebras of finite posets, we give one of them below. 

\begin{definition}[incidence algebra]
Let $(\Delta,\preceq)$ be a poset with $n$ elements. The incidence algebra $K\Delta$ is a quotient of the path algebra of the following quiver $Q$. 
The set of vertices, $Q_0$, is in bijection with the elements of the poset $\Delta$ and the set of arrows $Q_1$ is defining by declaring that there is an arrow $\alpha$ from a vertex $a$ to a vertex $b$, whenever $a \preceq b$ and there is no $a \preceq c \preceq b$, with $c \neq a$ and $c \neq b$.
Let $I$ be the ideal generated by all commutativity relations $\gamma - \gamma'$, with $\gamma$ and $\gamma'$ parallel paths. 
The incidence algebra $K\Delta$ is $KQ/I$.
\end{definition}

The quiver $Q$ of the incidence algebra, in the former definition, is also called the Hasse quiver of the poset.

We are going to assume always that our incidence algebras are connected, that is the Hasse quiver is connected.

For more details in the subject of incidence algebras we refer to \cite{lou} and \cite{bau-vil}. 

We want to define next the notion of piecewise hereditary algebras. In order to do this we need to introduce, very briefly, some previous notions. 

Given an abelian category $\mathcal{A}$ we denote $D^b(\mathcal{A})$ its bounded derived category, as usual if $A$ is a $K$-algebra then $D^b(A)$ denotes the bounded derived category of $\md A$.
  
An abelian category $\mathcal{H}$ is called \emph{hereditary} if the extension groups $\Ext_\mathcal{H}^n (\!X,Y\!)$ are zero for all $n \geq 2$ for any pair of objects $X$ and $Y$ of $\mathcal{H}$.

\begin{remark}
All hereditary categories considered in this paper have splitting idempotents, finite dimension $\Hom$ spaces, and tilting object. See below the definition of tilting object.
\end{remark}

\begin{definition}[piecewise hereditary algebra]
We say that $A$ is piecewise hereditary algebra of type $\mathcal{H}$ if there exists a hereditary abelian category $\mathcal{H}$, with splitting idempotents, finite dimension $\Hom$ spaces,
such that $D^b(\mathcal{A})$ is triangle-equivalent to the bounded derived category $D^b(\mathcal{H})$.
\end{definition}

For more details in the subject of piecewise hereditary algebra we refer to \cite{hap-rei-sma}, \cite{che-kra}, \cite{hap-ric-scho}, \cite{hap2}, \cite{lad2}, \cite{len}, \cite{len-sko}, \cite{bar-len1}, \cite{bar-len2}, \cite{hug-koe-liu}, \cite{hap-zac2}, \cite{meu} and \cite{alv-lem-mar}.

The definition of tilting modules inspired the definition of \emph{tilting object} that follows:

\begin{definition}[tilting object]
Let $\mathcal{H}$ be a hereditary abelian $K$-category. An object $T \in \mathcal{H}$ is called tilting if
\begin{enumerate} [\itshape i)]
\item $\Ext^1_{\mathcal{H}} (T,T)=0$, and
\item for every $X \in \mathcal{H}$ the condition $\Ext^1_{\mathcal{H}} (T,X)=0=Hom_{\mathcal{H}}(T,X)$ implies that $X=0$.
\end{enumerate}
\end{definition}

Let $A$ piecewise hereditary algebra of type $\mathcal{H}$. It follows from Rickard's theorem \cite{ric}, the existence of a tilting object $T$ in $D^b(\mathcal{H})$ such that $A = \End T$. 

Given a sequence  $p_1,\dotsc,p_n$ of positive integers, $\mathbb{X}(p_1,\dotsc,p_n)$, will denote the  weighted projective line of type $p_1,\dotsc,p_n$, in the sense of \cite{gei-len}, and $\Coh \mathbb{X}$  the category of coherent sheaves over $\mathbb{X}(p_1,\dotsc,p_n)$. 
Let $Q$ be a finite, connected quiver without oriented cycles and let $KQ$ denote the path algebra of $Q$.
We state one of the most important theorems about piecewise hereditary algebras.

\begin{teo}[Happel \cite{hap2}]
Let $\mathcal{H}$ be an abelian hereditary connected $K$-catego\-ry with tilting object. Then $\mathcal{H}$ is derived equivalent to $\md KQ$ or derived equivalent to $\Coh \mathbb{X}$ for some weighted projective line $\mathbb{X}$.
\end{teo}

An algebra $A$ is called a \emph{piecewise hereditary algebras of quiver type} (or of type $Q$) or \emph{of sheaf type} if $A \cong' KQ$ for some quiver $Q$ or $A \cong' \Coh \mathbb{X} $ for some weighted projective line $\mathbb{X}$, respectively.

Observe that an algebra can be, at the same time, of quiver and sheaf type.

\section{Simply connectedness} \label{sec3}
In 1993, in the article \cite{sko}, Skowro\'nski proposed the following:

\begin{quotation}
Describe classes of algebras for which is it true that if an algebra 
$A$ is in one of these classes then it is simply connected if and only if $HH^1(A)=0$?
\end{quotation}

Given a $K$-algebra $A$ we recall the definition of the first Hochschild cohomology group of $A$. This can also be done via a complex, defined by Hochschild, where all the cohomology groups are defined at the same time, but since we need only the first group, we decided to give an ad hoc definition.

Let $A$ a $K$-algebra then a derivation of $A$ is a $K$-linear endomorphism of $A$, $\alpha$ such that $\alpha(ab) =a\alpha(b) + \alpha(a)b$, for all pair of elements $a,b$ in $A$. The set of 
derivations, $Der(A)$ form a $K$-subspace of $\End_K (A)$.  Given an element $x \in A$ we can define a derivation, denoted by $\mathcal{A}dd_x$ using the formula $\mathcal{A}dd_x(a) = ax -xa$, for all $a \in A $, such derivation is called inner derivation. The set of inner derivations, $\mathcal{I}nn(A)$ form a subspace and the first Hochschild cohomology of $A$ is the quotient $Der(A) / \mathcal{I}nn(A)$. 

We will show, in this section, the following statement. In the statement we have a restriction which we believe it is not necessary, but we where not able to show
the result without this restriction. 
\begin{teo*}
Let $K\Delta$ be a PHI algebra, which is not of type a wild quiver,  then $HH^1(K\Delta)=0$ if and only if $K\Delta$ is simply connected.
\end{teo*}
The implication $K\Delta$ is simply connected then $HH^1(K\Delta)=0$ has already been proved for incidence algebras in general, De la Pe\~na and Saor\'in showed the following result:
\begin{teo}[De la Pe\~na, Saor\'in \cite{pen-sao}]
Let $K\Delta$ be an incidence algebra and $(Q,I_v)$ a presentation of $K\Delta$. Then
$$
\Hom_{gr} (\pi_1 (Q , I_v ) , K^+) \cong HH^1 (K\Delta).
$$
\end{teo}
Here  $\Hom_{gr} (\pi_1 (Q , I_v ) , K^+)$ is the group of all group homomorphisms from the group $\pi_1$ to the additive group $K^{+}$, where $K^{+}$ denotes the additive group of the field. 

For a definition of the homotopy group $\pi_1(Q, I_v)$ see for instance \cite{pen-vil, pen-sao, far-gei-gre-mar}.

It should be noted that for the implication ``If $HH^1(K\Delta)$ $=0$ then $K\Delta$ is simply connected'' we need the hypothesis that $K\Delta$ is a PHI algebra. The following is a counterexample, showing that this is not valid, in general.
Consider the projective plane, whose triangulation has the simplicial complex with the fundamental group isomorphic to $\mathbb{Z}_2$. Knowing that the fundamental group of the simplicial complex is isomorphic to the fundamental group of poset $\Delta$ associated with this complex, applying the previous theorem we get:
$$
\Hom (\mathbb{Z}_2 , K^+) \cong HH^1 (K\Delta).
$$
Then $K\Delta$ is not simply connected but $HH^1 (K\Delta)=0$, if the characteristic of $K$ is not 2.

The main theorem is in article ``Topological invariants of piecewise hereditary algebras'' \cite{meu} with the following statement:
\begin{teo}[Le Meur \cite{meu}]
Let $A$ be a connected algebra derived equivalent to a hereditary abelian category $\mathcal{H}$ whose oriented graph $\mathcal{K}_\mathcal{H}$ of tilting objects is connected. The following are equivalent:
\begin{enumerate}[\itshape a)]
\item $HH^1(A) = 0$.
\item $A$ is simply connected.
\end{enumerate}
\end{teo}

In order to use this we need the fact that the oriented graph of tilting objects (defined below) is connected.

\begin{definition}[tilting graph \cite{hap-ung}]
The oriented graph $\mathcal{K}_\mathcal{H}$ of tilting objects of a category $\mathcal{H}$ has vertex set in bijection with the isoclasses of the tilting objects. Let $T$,$T'$ be  
non isomorphic tilting objects, there exist an arrow $T' \to T$ if $T'=U \oplus Y$, $T=U \oplus X$ with $X$,$Y$ are non-isomorphisms indecomposables and there is a short exact sequence
$$
0 \longrightarrow Y \longrightarrow \tilde{U} \longrightarrow X \longrightarrow 0
$$
with $\tilde{U} \in \add U$.
\end{definition}

When $\mathcal{H}$ is derived equivalent to a hereditary algebra which is not of wild type, Happel and Unger \cite{hap-ung} decided on  the connectedness of  $\mathcal{K}_\mathcal{H}$ with the following result:
\begin{teo}[Happel-Unger \cite{hap-ung}]
Let $KQ$ tame hereditary algebras. The $\mathcal{K}_{KQ}$ is connected if and only if $Q \neq \widetilde{A}_{1,p}$.
\end{teo}
We observe that $HH^1 (K\widetilde{A}_{1,p}) \equiv \mathbb{K}$ then the PHI algebras of type $\widetilde{A}_{1,p}$ are not simply connected.

Barot, Kussin and Len\-zing \cite{bar-kus-len} proved the connectedness of $\mathcal{K}_\mathcal{H}$ provided that $\mathcal{H} \cong' \Coh \mathbb{X}$ for weighted projective line $\mathbb{X}$ of tubular type.

Recently on the work \cite{fu-gen}, Fu and Geng proved the following result:
\begin{teo}[Fu-Geng \cite{fu-gen}]
Let $\mathcal{H}$ be a connected hereditary abelian category over $K$. The tilting graph $\mathcal{K}_\mathcal{H}$ is connected provided that $\mathcal{H}$ does not contain nonzero projective objects.
\end{teo}
%\cite{bmrrt}

Therefore, together with the result of Le Meur \cite{meu}, we can statement:

\begin{teo} \label{h1-sc}
Let $K\Delta$ be a PHI algebra that is not of wild-quiver type. If $HH^1(K\Delta)=0$, then 
the algebra $K\Delta$ is simply connected.
\end{teo}

\section{Global dimension} \label{sec4}
The \emph{projective dimension} of an object $M$ in an abelian category $\mathcal{A}$ is by definition 
$$
\dpr_\mathcal{A} M = \sup \{d \in \mathbb{N} \: / \: \Ext_\mathcal{A}^{d} (M,M') \neq 0 \text{ for some } M' \}.
$$
and the \emph{global dimension} of $\mathcal{A}$ is the $\sup\{\dpr_\mathcal{A} M : M \text{ is an object in } \mathcal{A} \}$.

When we refer to the global dimension of algebras $A$, we are considering the global dimension of the category $\md A$. The piecewise hereditary algebras have finite global dimension. We mention the following result:

\begin{teo}[Happel-Reiten-Smal\o\ \cite{hap-rei-sma}]
Let $\mathcal{A}$ be a piecewise hereditary, abelian category with finite length and $n$ non-isomorphisms simples objects. Then the category 
$\mathcal{A}$ has global dimension less or equal to $n$. 
\end{teo}

We asked ourself the following question: Is there an upper bound for the global dimension of  PHI algebras? 

This question was answered by Ladkani \cite{lad2}.  To state the result of Ladkani we recall that an algebra is called sincere if it admits a 
sincere indecomposable module in its category of modules, that is, an indecomposable module such that every simple module is a composition factor of it. 
The following statement is a particular case of Ladkani's statement:

\begin{teo}\cite{lad2}
A sincere, piecewise hereditary algebra has global dimension less or equal to $3$.
\end{teo} 

Connected incidence algebras are sincere, as we see next.

\begin{prop} \label{sincero}
Let $K\Delta = KQ_{K\Delta}/I$ be an incidence algebra. Then $K\Delta$ is a sincere algebra.
\end{prop}
\begin{proof}
We need to show the existence of an indecomposable, sincere, module $M$ over $K\Delta$. The candidate $M$ is the  module associated with the 
following representation:
\begin{enumerate} [\itshape a)]
\item for each vertex $a$ in $Q_{K\Delta}$ we associate $K$;
\item for each arrow $\alpha \colon a \to b$ in $Q_{K\Delta}$ We associate the identity $1 \colon K \to K$. 
\end{enumerate}
First, we will show that $M$ is indecomposable. For this, we will study the $\End M$. We consider $f = (f_a)_{a \in Q_0}$ a non-zero morphism of $\End M$. Thus, there exists $f_a \colon K \to K $ not zero for some $a \in Q_0$, implying that $f_a$ is an isomorphism. Given an arrow $\alpha \colon a \to b $, we have that $f_a 1 = 1 f_b$. If the arrow is in the other direction, we get the same result. Thus, no matter the direction of the arrow, we 
conclude that $f_a = f_b$. So, since the graph is connected, we always have a walk connecting the vertex $a$ to any vertex $c$:
$$
\xymatrix@1{
a \ar@{-}[r] &\bullet \ar@{-}[r] &\dotsb \ar@{-}[r]& c
}
$$
By a finite process, we conclude that $f_a = f_c$ for all vertex $c$ of $Q_{K\Delta}$ and consequently $\End M \cong K $, therefore $M$ is an indecomposable module.

It is clear that $M$ is sincere. 

Therefore, $M$ is a sincere indecomposable module over an incidence algebra $K\Delta$. 
\end{proof} 

As a consequence we get the following corollary:

\begin{cor}
The global dimension of a PHI algebra is less or equal to $3$.
\end{cor}

\begin{definition}
We call the module $M$ in the proposition above the \emph{canonical sincere module}. 
\end{definition}

The representation-finite sincere algebras have global dimension less or equal to two. Before proving this 
affirmation, we need the definition of a directed module.

\begin{definition}[directed module]
A cycle in the module category is a sequence 
$$
M_0 \xrightarrow{f_1} M_1 \xrightarrow{f_2} M_2 \xrightarrow{f_3} \dotsc \rightarrow M_{t -1} \xrightarrow{f_t} M_t \cong M_0 \: ,
$$
where $t\geq 1$, each $M_i$ is indecomposable and each morphism  $f_i$ is non-zero, and non isomorphism.

Let $M$ be an indecomposable module, $M$ is called directed if it does not belong to any cycle.
\end{definition} 

Happel showed in the article ``On the derived category of a finite-dimen\-sional algebra'' the following result:
\begin{cor}[Happel \cite{hap3}] \label{Adirec}
Let $A$ be a representation-finite, piecewise hereditary algebra. Then $\md{A}$ is directed, that is, all indecomposable $A$-modules are directed. 
\end{cor}

We say that $A$ is directed if the category $\md A$ is directed.

Observe that for a directed algebra, over an algebraically closed field $K$, the endomorphism ring of an indecomposable module is isomorphic to $K$. since if there is a non zero endomorphism which is not an isomorphism, we get a cycle of lenght $1$.

Now, we use the following Ringel theorem \cite{rin}:
\begin{teo}[Ringel \cite{rin}]
Let $A$ be an algebra having a sincere and directed indecomposable module. Then $A$ is a tilted algebra.
\end{teo}

The following result is a consequence of the two previous statements.

\begin{prop}
Let $A$ a representation-finite sincere algebra. Then $A$ is a tilted algebra, consequently $\gl{A} \leq 2$.
\end{prop}

\begin{cor}
If $A$ is a PHI algebra of finite representation type they its global dimension is less or equal to $2$.
\end{cor}

The global dimension of an incidence algebra is related with it strong simply connectedness. Before we show some results 
in this direction, we will introduce a family of algebras with global dimension equal to three called critical algebras.

\begin{definition}[critical algebra \cite{bor-fern-tre}] 
Let $B$ be an algebra. We say that $B$ is critical if it satisfies the following properties:
\begin{enumerate}[(i)]
\item $B$ has a unique source $i$ and a unique sink $j$,
\item Let $S_a$ be a simple module associated with source $a$ and let $S_b$ be a simple module associated with sink $b$, then $\dpr{S_a} = 3$ and 
$\din{S_b} = 3$. If $S$ is a simple module associated with other vertice then $\dpr{S} \leq 2$ and $\din{S} \leq 2$.
\item Consider the minimal projective resolution of the simple module $S_a$:
$$
0 \longrightarrow P_3 \longrightarrow P_2 \longrightarrow P_1 \longrightarrow P_0 \longrightarrow S_a \longrightarrow 0.
$$
Let $P$ be the following projective module, $P=\bigoplus_{k=0}^3 P_k$. Then all indecomposable projective are in $\add{P}$, and each indecomposable projective is a direct 
summand of exactly one $P_k$, for $k \in \{0,\dotsc,3\}$.
\item Consider the  minimal injective resolution of the $S_b$:
$$
0 \longrightarrow S_b \longrightarrow I_0 \longrightarrow I_1 \longrightarrow I_2 \longrightarrow I_3 \longrightarrow 0.
$$
Consider the  $B$-module $I=\bigoplus_{k=0}^3 I_k$. Then all indecomposable injectives are in $\add{I}$, and each indecomposable injective is a direct 
summand of exactly one $I_k$, for $k \in \{0,\dotsc,3\}$.
\item $B$ does not contain any proper full subcategory that verifies i), ii), iii) e iv).
\end{enumerate}
\end{definition}

A description by quivers and relations of all the critical algebras can be found in the work ``A criterion for global 
dimension two for strongly simply connected schurian algebras'' \cite{bor-fern-tre}.

\begin{prop}[Bordino-Fernandez-Trepode \cite{bor-fern-tre}]
Let $B$ be a critical algebra. Then the algebra has one of the following presentations.

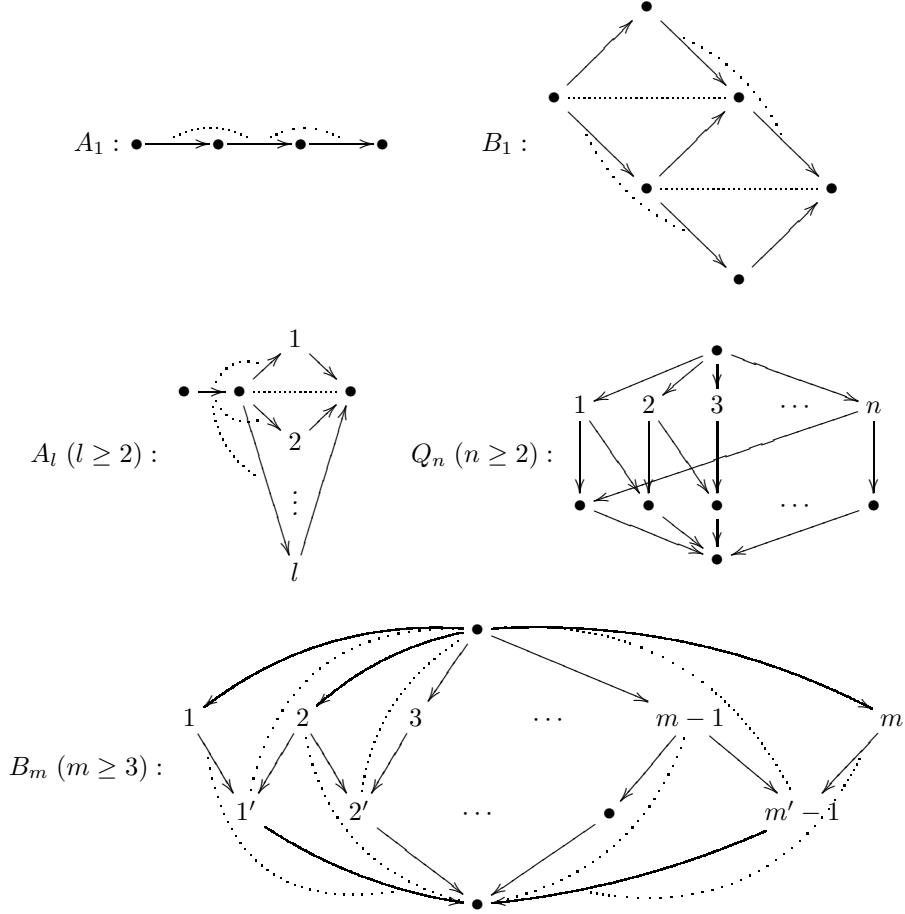
\begin{align*}
A_1:
\xymatrix@1{
\bullet \ar[r] ^(.3){}="a" & \bullet \ar[r] ^{}="b" & \bullet \ar[r] ^(.6){}="c" & \bullet \ar@{.} @/^/ "a";"b" \ar@{.} @/^/ "b";"c"
} && 
B_1:
\begin{aligned}
\xymatrix{
& \bullet \ar[dr] ^(.3){}="B" &&\\
\bullet \ar[ur] \ar[dr] ^(.3){}="A" \ar@{.}[rr] && \bullet \ar[dr] ^{}="D" & \\
& \bullet \ar[dr] ^{}="C" \ar@{.}[rr] \ar[ur] && \bullet \\
&& \bullet \ar[ur] \ar@{.} @/_/ "A";"C" \ar@{.} @/^/ "B";"D" &
}
\end{aligned}
\end{align*}
\begin{align*}
A_l \; (l \geq 2):
\begin{aligned}
\xymatrix@R8pt@C10pt{
&& 1 \ar[dr] 	&\\
\bullet \ar[r] ^{}="A" _{}="E" & \bullet \ar[ur] ^{}="B" \ar[dddr] _{}="D" \ar[dr] _{}="C" \ar@{.}[rr] && \bullet \ar@{.} @/^/ "A";"B" \ar@{.} @/_/ "E";"C" \ar@{.} @/_/ "E";"D"\\
&& 2 \ar[ur]	&\\
&& \vdots		&\\
&& l \ar[uuur]	&
}
\end{aligned}
&& 
Q_n \; (n \geq 2):
\begin{aligned}
\xymatrix@R10pt@C15pt{
&& \bullet \ar[dll] \ar[dl] \ar[d] \ar[drr] &&\\
1 \ar[dd] \ar[ddr] & 2 \ar[dd] \ar[ddr] & 3 \ar[dd] & \cdots & n \ar[dd] \ar[ddllll] \\
&&&&\\
\bullet \ar[drr] & \bullet \ar[dr] & \bullet \ar[d] & \cdots & \bullet \ar[dll] \\
&& \bullet &&
}
\end{aligned}
\end{align*}
\begin{align*}
B_m \; (m \geq 3):
\begin{aligned}
\xymatrix@R8pt@C9pt{
&&&&& \bullet \ar@/_1pc/[ddlllll] \ar@/_/[ddlll] \ar@{.}@/_4mm/[ddddll] \ar[ddl] \ar[ddrrr] \ar@/^1pc/[ddrrrrr] \ar@{.}@/_9mm/[ddddllll] \ar@{.}@/^9mm/[ddddrrrr] &&&&&\\
&&&&&&&&&&\\
1 \ar[ddr] ^(.3){}="A" && 2 \ar[ddl] \ar@{.}@/_6mm/[ddddrrr] \ar[ddr] && 3 \ar[ddl] && \cdots && m-1 \ar[ddl] \ar@{.}@/^6mm/[ddddlll] \ar[ddr] && m \ar[ddl] ^(.3){}="C" \\
&&&&&&&&&&\\
& 1' \ar@/_/[ddrrrr] ^(.7){}="B" && 2' \ar[ddrr] && \cdots  && \bullet \ar[ddll]&& m'-1 \ar@/^/[ddllll] ^(.7){}="D"	&\\
&&&&&&&&&&\\
&&&&& \bullet \ar@{.} @/_2pc/ "A";"B" \ar@{.} @/^2pc/ "C";"D" &&&&&
}
\end{aligned}
\end{align*}

In the case $Q_n$, above, the relations are given by declaring that two parallel paths are equal.
\end{prop}

Using this proposition, we see that the only critical incidence algebras are the ones whose quiver is the $Q_n$, since
the presentations of algebras $A_1$, $A_l$, $B_1$ and $B_l$ have non-commutative relations.

Now we can state a theorem of Bordino, Fern\'andez and Trepode \cite{bor-fern-tre} which we will use in our proposition \ref{dimglFSC}.

\begin{teo}[Bordino-Fern\'andez-Trepode \cite{bor-fern-tre}]
Let $A$ be a strongly simply connected schurian algebra with global dimension greater or equal to three. Then there exists a full subcategory $B$ 
of $A$ such that $B$ is critical.
\end{teo}

Let $A$ be a schurian triangular algebra, the interval $[x, y]$ between $x$ and $y$ is the full subcategory of $A$ generated by all points $z \in A$ which lie on a nonzero path from $x$ to $y$, that is, such that $xAz \; zAy \neq 0$.
 
In order to state our next theorem we need the definition of crown.
 
\begin{definition}[crown \cite{ass-pla-red-tre}]
Let $C$ be a full subcategory of $A$ generated by $2n$ points $\{x_1,\ldots,x_n,y_1,\ldots,y_n\}$, with $n \geq 2$, and of the form:
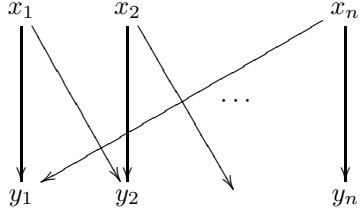
\begin{gather*}
\xymatrix{
x_1 \ar[dd] \ar[ddr]& x_2 \ar[dd] \ar[ddr] 	& 				& x_n \ar[dd] \ar[ddlll] \\
					&						& \cdots		& \\
y_1 				& y_2					& 	& y_n
}
\end{gather*}
\begin{enumerate}[\itshape i)]
\item We say that $C$ is a weak crown in $A$ if:
\begin{enumerate}
\item For each $i$, $[x_i,y_i]$ intersects those of $[x_{i-1},y_i]$ and
$[x_i,y_{i+1}]$, and of no other $[x_h,y_l]$ (here, and in the sequel, we agree to set
$y_{n+1}=y_1$ and $x_0=x_n$).
\item the intersection of three distinct $[x_h,y_l]$ is empty.
\end{enumerate}

\item A weak crown $C$ is said to be a crown if, for each $i$, the intersection of $[x_i,y_i]$ and of $[x_i,y_{i+1}]$ is $x_i$, and the intersection of $[x_{i-1},y_i]$ and of $[x_i,y_i]$ is $y_i$.
\end{enumerate}
\end{definition}

Before stating the proposition \ref{dimglFSC}, we need the following result:

\begin{teo} [Assem-Castonguay-Marcos-Trepode \cite{ass-cas-mar-tre}]
Let $K\Delta$ be an incidence algebra, $K\Delta$ is strongly simply connected if and only if $K\Delta$ does not contain a crown.
\end{teo}

Now we have the result on incidence algebras. 

\begin{prop} \label{dimglFSC}
Let $K\Delta$ be a strongly simply connected incidence algebra. Then the global dimension of $K\Delta$ is less or equal to two.
\end{prop}
\begin{proof}
We recall that an incidence algebra is a schurian algebra. If the global dimension of $K\Delta$ is greater or equal to three then it will contain 
a full subcate\-gory $B$ such that $B$ is critical. Since $K\Delta$ is an incidence algebra, $B$ is of the form $Q_n$ and thus, the incidence algebra $K\Delta$ 
would contain 
a crown. This contradicts the result above.
\end{proof}

\section{Strong global dimension} \label{sec5}

Let $\mathcal{P}_A$ be the full subcategory of $\md A$ consisting of the projective modules. We denote $C^b(\mathcal{P}_A)$ the category of 
bounded complexes with entries in  $\mathcal{P}_A$ and $K^b(\mathcal{P}_A)$ the corresponding homotopy category.

\begin{definition}
 A complex $P=(P^i, d^i)$ in $C^b(\mathcal{P}_A)$ is called radical if the image of each differential $d^i$ is contained in the radical of $P^{i-1}$.
\end{definition}

The following proposition is well known.

\begin{prop}
 Every complex $P=(P^i, d^i)$  in  $C^b(\mathcal{P}_A)$ is homotopic to a unique, up to isomorphism, radical complex.
\end{prop}

Due to the former proposition when we consider a  complex in the category $K^b(\mathcal{P}_A)$ we will assume that it is radical, since it is isomorphic in $K^b(\mathcal{P}_A)$ to a unique radical complex.

In reality what happens is that the full subcategory of $K^b(\mathcal{P}_A)$ whose objects are the radical complexes is equivalent to the category $K^b(\mathcal{P}_A)$.

 \begin{definition}

If $P=(P^i, d^i) \in 
K^b(\mathcal{P}_A)$ is a radical complex, which is not zero, then there are integers $r \leq s$ such that $P^i=0$ for all $i<r$ and $P^i=0$ for all $s<i$, with $P_r$ and $P_s$ not zero. 
The length of $P$ is defined as $\operatorname{l}(P)=s-r$.
\end{definition}

We can now define the strong global dimension.

\begin{definition}[strong global dimension \cite{hap-zac}]
Let $A$ a finite-dimensional algebra. The strong global dimension of $A$ is defined in the following way.
$$
\glf A = \sup \{\operatorname{l}(P) \: | \: P \in K^b(\mathcal{P}_A)\; \text{ indecomposable}\}.
$$
\end{definition}

Observe that the strong global dimension can be infinite.

The next theorem is important in the study of the strong global dimension for PHI algebras.
\begin{teo}[Happel-Zacharia \cite{hap-zac}]
Let $A$ a finite-dimensional algebra. The algebra is piecewise hereditary if and only if $\glf A$ is finite.
\end{teo}

Thus it is clear that PHI algebras have finite strong global dimension. A tool to compute the strong global dimension of an algebra is the following theorem.
\begin{teo}[Alvares, Le Meur, Marcos \cite{alv-lem-mar}]
Let $\mathcal{T}$ be triangulated category which is triangle equivalent to the bounded derived category of a hereditary abelian category. Let $T \in \mathcal{T}$ a tilting object. 

There exists a full and additive subcategory $\mathcal{H} \subset \mathcal{T}$ which is hereditary and abelian, such that the embedding $\mathcal{H} \hookrightarrow \mathcal{T}$ extends to a triangle equivalence $D^b(\mathcal{H}) \cong \mathcal{T}$, and
$$
T \in \bigvee_{m=0}^l \mathcal{H}[m]
$$
for some integer $l \geq 0$. Moreover, if $(\End_\mathcal{T} T)^{op}$ is not hereditary then there exists such a pair $(\mathcal{H},l)$ verifying $\glf 
(\End_\mathcal{T} T)^{op} = l+2$
\end{teo}

Inspired by article the ``Jordan H\"older theorems for derived module categories of piecewise hereditary algebras'' \cite{hug-koe-liu}, Alvares, Le Meur and 
Marcos have used an alternative definition of the strong global dimension \cite{alv-lem-mar}. 
\begin{lema}[Alvares, Le Meur, Marcos \cite{alv-lem-mar}]
Let $T \in D^b(\mathcal{H})$ be a tilting object such that $A = (\End T)^{op}$. Given a object $Y$ in $D^b(\mathcal{H})$, we define
\begin{equation*}
\begin{split}
\ell^{+}_T(Y) &= \sup \{n \in \mathbb{Z} \: | \: \Hom_{D^b({\mathcal{H}})}(Y,T[n]) \neq 0 \} \\
\ell^{-}_T(Y) &= \inf \{n \in \mathbb{Z} \: | \: \Hom_{D^b({\mathcal{H}})}(T[n],Y) \neq 0 \}.
\end{split}
\end{equation*}
Then the strong global dimension of $A$ is  $\glf A = \sup \{ \ell^{+}_T(Y) - \ell^{-}_T(Y) \: | \: Y \in D^b(\mathcal{H}) \}$.
\end{lema}
We want to give an upper bound for the strong global dimension of the PHI algebras. We have a more general result.

\begin{teo}
Let $A$ be a sincere, piecewise hereditary algebra. Then 
$$\glf A \leq 3 \; .$$
\end{teo}
\begin{proof} 
We consider $M$ a sincere module and $A=P_1 \oplus \dotsc \oplus P_n$ the decomposition of $A$  in indecomposable modules. 

By hypothesis, there exist a quasi-inverse functor $F \colon D^b(A) \to D^b(\mathcal{H})$ which makes the triangulated equivalence, where $\mathcal{H}$ is a hereditary category.

Since $M$ is sincere, we have that
$$
0 \neq \Hom_A (P_i,M) = \Hom_{D^b(A)} (P_i,M) \cong \Hom_{D^b(\mathcal{H})} (FP_i,FM). 
$$

Let $FA=T$ be a tilting object in $D^b(\mathcal{H})$ such that $(\End T)^{op} \cong A$. For each $i \in \{1,\dotsc ,n\}$, we denote $FP_i=T_i[m_i]$ the indecomposable direct summand of $T$ where $T_i \in \mathcal{H}[0]$ such that $D^b(\mathcal{H})= \bigvee_{m \in \mathbb{Z}} \mathcal{H}[m]$. Also we denote $FM=M'[m]$ where $M' \in \mathcal{H}[0]$.

Now, for each $i \in \{1,\dotsc ,n\}$, we have the following:
$$
0 \neq \Hom_{D^b(\mathcal{H})} (FP_i,FM) = \Hom_{D^b(\mathcal{H})}(T_i[m_i],M'[m]).
$$ 
Therefore $m_i=m$ or $m_i=m-1$, for each $i \in \{1,\dotsc ,n\}$, implying that the indecomposables direct summands of $T$ are in $\mathcal{H}[m-1]$ or $\mathcal{H}[m]$ which  shows that $\glf A \leq 3$.
\end{proof}

As an immediate consequence of previous theorem, we get the following:
\begin{cor}
Let $K\Delta$ be a PHI algebra . Then $\glf K\Delta \leq 3$.
\end{cor}

Another corollary is what has already been proved by Ladkani in \cite{lad2}. 
\begin{cor}
Let $A$ be a sincere, piecewise hereditary algebra. Then 
$$
\gl A \leq 3 \; .
$$
\end{cor}
\begin{proof} 
This is a consequence of inequality $\gl A \leq \glf A$.
\end{proof}

\begin{example}
The incidence algebra, whose quiver is given below, is a PHI algebra which has global dimension equal to three, and therefore it also has strong global dimension three.
$$
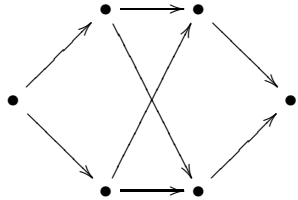
\xymatrix{
						&\bullet \ar[ddr] \ar[r]&\bullet \ar[dr]&\\
\bullet \ar[ur] \ar[dr]	&						&				&\bullet \\
						&\bullet \ar[uur] \ar[r]&\bullet \ar[ur]&
}
$$
\end{example}

Our result \ref{dimglFSC} implies that  the family of PHI algebras that has  global dimension and strong global dimension equal to three are not strongly 
simply connected. 

It would be interesting to give a characterization of this family of algebras? Observe that the PHI algebra above is in this class.

\section{PHI algebras of sheaf type} \label{Phiasfeixe}
In 1987, Geigle and Lenzing introduced category of coherent sheaves on the weighted projective line 
$\mathbb{X}$\ \cite{gei-len}, denoted by  $\Coh \mathbb{X}$. This is an abelian hereditary category which is
derived equivalent to category of modules over a canonical algebra $C(p,\lambda)$ \cite{gei-len}. Some basic references to this subject are  ``Tame 
algebras and integral quadratic forms'' \cite{rin} and ``Elements of the representation theory of associative algebras, volume three'' \cite{sim-sko3}.

%\begin{definition}[Canonical algebra]
%Let $K$ be an algebraically closed field. 
%For a sequence of positive integers $p = (p_1,\cdots, p_m)$ and a sequence of parameters $\lambda = (\lambda_3, \cdots, \lambda_m)$ we denote by $C(p,\lambda)$ the canonical k-algebra which is given by the quiver
%
%\begin{center}
%\begin{tikzcd}
%&\bullet_{1,1} \ar{r}{\alpha_{12}} &\bullet_{1,2} \ar{r}{\alpha_{13}} &\cdots \ar{r}{\alpha_{1p_1 \menas 1}} &\bullet_{1,p_1 \menas 1} \ar{ddr}{\alpha_{1p_1}} &\\
%&&&&&\\
%\bullet_0 \ar{r}{\alpha_{21}} \ar{uur}{\alpha_{11}} \ar[swap]{ddr}{\alpha_{m1}}	& \bullet_{2,1} \ar{r}{\alpha_{22}} &\bullet_{2,2} \ar{r}{\alpha_{23}} &\cdots \ar{r}{\alpha_{2p_2 \menas 1}} &\bullet_{2,p_2 \menas 1} \ar{r}{\alpha_{2p_2}} &\bullet_w \\
%&\vdots &\vdots &\vdots &\vdots &\\
%&\bullet_{m,1} \ar{r}{\alpha_{m2}} &\bullet_{m,2} \ar{r}{\alpha_{m3}} &\cdots \ar{r}{\alpha_{mp_m \menas 1}} &\bullet_{m,p_m \menas 1} \ar[swap]{uur}{\alpha_{mp_m}}&
%\end{tikzcd}
%\end{center}
%
%and the relations 
%$$
%\alpha_{jp_j}\cdots\alpha_{j1} + \alpha_{1p_1}\cdots\alpha_{11} + \lambda_j \alpha_{2p_2}\cdots\alpha_{21} \text{ com } j \in \{3,\cdots,m\}.
%$$
%\end{definition}

The article ``A class of weighted projective curves arising in representation theory of finite dimensional algebras'' \cite{gei-len} 
by Geigle and Lenzing, and the paper ``Introduction to coherent sheaves on weighted projective line'' \cite{che-kra} by Chen and Krause are important 
texts for an introduction to the theory of category of coherent sheaves on weighted projective line.  We will use the characterization of the category 
$\Coh \mathbb{X}$ exposed in ``Hereditary noetherian categories with a tilting complex'' \cite{len} by Lenzing.

The PHI algebras $K\Delta \cong' \Coh \mathbb{X}$ were studied in the article ``Which canonical algebras are derived equivalent to incidence algebras of posets?'' \cite{lad1} of Ladkani, which in part, relates to our work.
This article was published in 2008, and presents a direct study of PHI algebras of sheaves type or equivalently PHI algebras of canonical type, 
nomenclature influenced by the derived equivalence between the categories $\Coh (\mathbb{X})$ and the category of modules over a canonical algebra $C(p,\lambda)$. The main theorem of this article follows:

\begin{teo}[Ladkani \cite{lad1}]
Let $K\Delta$ be an incidence algebra. If $C(p,\lambda) \cong' K\Delta$ then $p=(p_1,p_2,p_3)$ or $p=(p_1,p_2)$, where $p_i \geq 2$. For 
each weighted projective line with $p=(p_1,p_2,p_3)$ or $p=(p_1,p_2)$, where $p_i \geq 2$, there exist at least one  incidence algebra $K\Delta$ such that $C(p,\lambda) \cong' K\Delta$.
\end{teo}

\begin{definition}[Euler Characteristic, domestic, tubular and wild]
\mbox{}
\begin{enumerate}

\item
The Euler characteristic $\mathcal{X}(\mathbb{X})$ of a weighted projective line is defined by the formula:
$$
\mathcal{X}(\mathbb{X})= 2 - \sum_{i=1}^n (1- \frac{1}{p_i}).
$$

\item The category $\Coh (\mathbb{X})$ is called domestic, tubular or wild, if its Euler characteristic is respectively bigger than zero, equal zero or smaller than zero. 
\end{enumerate}
\end{definition}

An algebra $A$ is called a \emph{piecewise hereditary algebra of domestic-sheaf type} or \emph{of tubular-sheaf type} or \emph{of wild-sheaf type} 
if $A \cong' \Coh (\mathbb{X})$ where $\Coh (\mathbb{X})$ is of domestic, tubular  or  wild type, respectively.

For the purpose of exhibiting some families of PHI algebras of sheaves type, 
we will study one-point extension. 
The papers of Barot, De la Pe\~na and Lenzing (\cite{bar-len1}, \cite{bar-len2} and \cite{len-pen}) provide conditions on the piecewise hereditary algebras $A$ and in the $A$-modules $M$ in order that $A[M]$ also results a piecewise hereditary. 

The next example is an example of a one point extension of a PHI algebra by its canonical sincere module, in which the one point extension is also PHI. 
It is clear that the one point extension, in the example is an incidence algebra, and we postponed the proof that it is also piecewise hereditary to \ref{phiawild}.
\begin{example}
Let $M$ be the canonical sincere module (see definition \ref{sincero}) over the algebra described below, by quiver and relation, and consider the one-point extension $K\Delta[M]$.

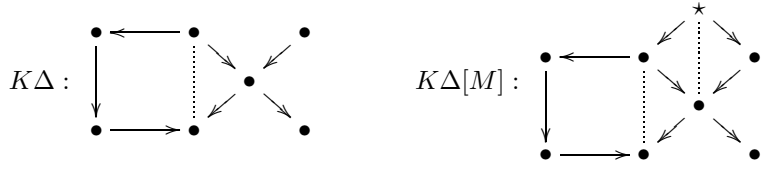
\begin{align*}
K\Delta :
\begin{aligned}
\xymatrix@R8pt@C10pt{
\bullet \ar[dd]	&& \bullet \ar[ll] \ar[dr]	&							& \bullet \ar[dl]\\
				&&							& \bullet \ar[dl] \ar[dr]	&\\
\bullet \ar[rr] && \bullet \ar@{.}[uu]		&							&\bullet
}
\end{aligned}
&&
K\Delta[M] :
\begin{aligned}
\xymatrix@R8pt@C10pt{
				&& 							& \star \ar[dl] \ar[dr] \ar@{.}[dd]	&\\
\bullet \ar[dd]	&& \bullet \ar[ll] \ar[dr]	&									& \bullet \ar[dl]\\
				&&							& \bullet \ar[dl] \ar[dr]			&\\
\bullet \ar[rr] && \bullet \ar@{.}[uu]		&									&\bullet
}
\end{aligned}
\end{align*}
We observe that $K\Delta[M]$ is an incidence algebra that has a poset with a unique maximal element represented by the vertex $\star$.
\end{example}

We denote $\Coh (\mathbb{X})_+$ (resp. $\Coh (\mathbb{X})_0$) the full subcategory of all vector bundles (resp. sheaves of finite length) on $\mathbb{X}$. 
%We are going to use the existence of a rank function $\rk$ on the Grothendieck group $K_0 ()=K_0 (D^b (\mathcal{C}))$ such that for an indecomposable $X \in \mathcal{C}$ we have $X \in \mathcal{C}_+$ (resp. $X \in \mathcal{C}_0$) if and only if $\rk X >0$ (resp. $\rk X =0$) \cite{gei-len}.

Assume that we have a triangular equivalence:
$$
F \colon D^b(K\Delta) \to D^b(\Coh (\mathbb{X})).
$$

Inspired by the work of Lenzing and Skowro\'nski, \cite{len-sko}, we will show in proposition, \ref{bundle}, that the canonical sincere module $M$ is associated, 
to a vector bundle, via the canonical equivalence on the derived categories.

%Lenzing and Skowro\'nski \cite{len-sko} decompose $T$ into a direct sum $T'_0[-1] \oplus T_+ \oplus T''_0$. By this, they identify $\md (A)$ with a full subcategory full of $D^b(\mathcal{C})$ whose objects $X$ satisfy $\Hom_{D^b(\mathcal{C})} (T,X[n])=0$ for all integer $n \neq 0$.
%By $\mathcal{C}$ to be hereditary, $\md (A)$ is contained in $\add (\mathcal{C}'[-1] \vee \mathcal{C}_+ \vee \mathcal{C}_0 \vee \mathcal{C}_+[1] \vee \mathcal{C}''_0[1])$. 

\begin{prop}[Lenzing-Skowro\'nski \cite{len-sko}]
Let $B$ be a representation-infinite, quasi-tilted algebra of sheaf type, obtained from the weighted projective line $\mathbb{X}$, via a tilting object $T$. 

Then each indecomposable $B$-module belongs to exactly one of the subcategories

\begin{enumerate} [\itshape a)]
\item $\md^l_0(B)$ consisting of all $Y[-1]$ where $Y$ in $\Coh (\mathbb{X})'_0$ satisfies \\ $\Hom_{\Coh (\mathbb{X})} (T_+,Y)=0$ and $\Ext^1_{\Coh (\mathbb{X})}(T'_0,Y)=0$;
\item $\md_+(B)$ consisting of all $Y$ from $T''^{\perp}_0 \cap \Coh (\mathbb{X})_+$ satisfying $\Ext^1_{\Coh (\mathbb{X})}(T_+,Y)=0$;
\item $\md^c_0(B)$ consisting of all $Y$ from $\Coh (\mathbb{X})_0$ satisfying $\Hom_{\Coh (\mathbb{X})} (T'_0,Y)=0$ and \\
$\Ext^1_{\Coh (\mathbb{X})}(T''_0,Y)=0$;
\item $\md_-(B)$ consisting of all $Z[1]$ with $Z$ in $T'^{\perp}_0 \cap \Coh (\mathbb{X})_+$ satisfying \\$\Hom_{\Coh (\mathbb{X})} (T_+,Z)=0$ ;
\item $\md^r_0(B)$ consisting of all $Z[1]$ with $Z$ in $\Coh (\mathbb{X})''_0$ satisfying $\Hom_{\Coh (\mathbb{X})} (T_+,Z)$ $=0$ and $\Hom_{\Coh (\mathbb{X})}(T''_0,Z)=0$.
\end{enumerate}
\end{prop}

\begin{prop} \label{bundle}
Let $A$ is a representation-infinite quasi-tilted algebra of sheaf type with a sincere indecomposable module $M$. 
Assume $F \colon D^b(A) \to D^b(\Coh (\mathbb{X}))$ 
is a quasi-invertible functor which defines a triangulated equivalence. 
Then $FM$ is a vector  bundle.
\end{prop}
\begin{proof}
We consider $A=P_1 \oplus \dotsc \oplus P_n$ the decomposition of $A$ in indecompo\-sa\-ble modules.
By hypothesis, $D^b (A)$ is equivalent as a triangulated category to $D^b (\Coh (\mathbb{X}))$, 
so there is a tilting object $T$ in $D^b (\Coh (\mathbb{X}))$ such that 
$$(\End_{D^b(\Coh (\mathbb{X}))} (T))^{op} = A .$$

Since $M$ is sincere, for each $i \in \{1,\dotsc ,n\}$ it follow:
$$
0 \neq \Hom_A (P_i,M) = \Hom_{D^b(A)} (P_i,M) \cong \Hom_{D^b(\mathcal{H})} (FP_i,FM). 
$$

Let $FP_i = Q_i[p]$ and $FM = Y[m]$, where $Q_i , Y$ belongs to $\Coh (\mathbb{X})[0]$, for each $i \in \{1,\dotsc ,n\}$. 
Therefore, 
$$
\Hom_{D^b(\Coh (\mathbb{X}))} (FP_i,FM) = \Hom_{D^b(\Coh (\mathbb{X}))} (Q_i[p], Y[m]) \neq 0.
$$
%Hence, $\Hom_{D^b(\Coh (\mathbb{X}))} (Q_i[p], Y[m])$ is not zero, for each $i \in \{1,\dotsc ,n\}$.

For each $i \in \{1,\dotsc ,n\}$, $Q_i[p]$ is an indecomposable direct summands of $T$, 
where $T$ decompose into a direct sum $T'_0[-1] \oplus T_+ \oplus T''_0$, see \cite{len-sko}. 
%Let $G$ be the quasi-inverse of $F$
%$$
%G \colon D^b(\Coh (\mathbb{X})) \to D^b(A) \quad \text{ such that } GT=A,
%$$
%we get that $FP_1 \oplus \dotsc \oplus FP_n$ is isomorphic to $T=T'_0[-1] \oplus T_+ \oplus T''_0$. 
That is, $Q_i[p]$ is isomorphic to an indecomposable direct summand of $T$ consequently $p \in \{-1,0\}$, 
and moreover $m=0$. 

We use the proposition of Lenzing-Skowro\'nski \cite{len-sko} above. 
If $FM \in \Coh (\mathbb{X})'_0$, we have that 
$$
\Hom_{D^b(\Coh (\mathbb{X}))} (T_+,FM) = \Hom_{\Coh (\mathbb{X})} (T_+,FM) \neq 0
$$
Then $FM[-1] \notin \md^l_0(B)$ eliminating the case a). 
Similarly, we eliminate the cases c) and e). 
This show that $FM$ is a bundle.

Case $A$ is a almost concealed-canonical algebra, see \cite{len-sko}, then $T=T_+ \oplus T''_0$. 
Implies that $p=0$ and $m \in \{0,1\}$, 
where $\Hom_{D^b(\Coh (\mathbb{X}))} (Q_i[p], Y[m]) \neq 0$. 
The case $FM=Y[0]$ is similar to the previous argument.
Let $FM=Y[1]$. If $FM \in \Coh (\mathbb{X})''_0$, we have 
$$
\Hom_{D^b(\Coh (\mathbb{X}))} (T_+, Y[1]) \cong \Ext_{\Coh (\mathbb{X})} (T_+,Y) \neq 0
$$
This contradicts that $T_+$ is a bundle and $Y$ is a finite length sheaf. 
Therefore $FM$ is bundle.

%We use the  proposition of Lenzing-Skowro\'nski \cite{len-sko} above. Take $x \in Q_0$ such that 
%$FP_x=P'_x[0]$ and $\Hom_{D^b(\mathcal{C})} (P'_x, FM) \neq 0$ then $FM$ is not isomorphic to $Y[-1]$ eliminating the case a). Analogously, by the 
%group of morphisms $\Hom_{D^b(\mathcal{C})} (P'_x, FM) \neq 0$ then $FM$ is not isomorphic to $Z[1]$ eliminating the cases d) and e). The cases b) 
%and c) such that $FM \cong Y$. We assume that the case c) is valid to the module $M$. Then $\Hom_{\mathcal{C}} (T'_0, FM)=0$, that is, 
%$\Hom_{D^b(\mathcal{C})} (P'_x, FM)=0$ for some $x \in Q_0$ contradicting the fact that $M$ is sincere. Therefore, the case c) is not valid for $M$. 
%Hence $FM \in T''^{\perp}_0 \cap \mathcal{C}_+$, then $FM$ is bundle.
\end{proof}

The next step is to demonstrate that the canonical sincere module $M$, of a PHI algebra is exceptional. 
We recall the definition of exceptional object:

\begin{definition}[exceptional]
An object $E$ in a triangulated $K$-category $\mathcal{T}$ is called exceptional 
if $\End (E) = K$ and, $E$ not have auto-extensions, 
that is, $\Hom_{\mathcal{T}} (E,E[n])$ $= 0$ for each non-null integer $n$. 

Corresponding, let $A$ a finite dimensional $K$-algebra, 
the $A$-module $E$ is called exceptional if $\End (E) = K$, and $\Ext_{A}^1 (E,E) = 0$.
\end{definition}

Again, the next proposition is not specific to incidence algebras. 

\begin{prop} \label{excep}
If $A$ is a representation-infinite quasi-tilted algebra of domestic-sheaf type with a sincere indecomposable module $M$, 
then $M$ is exceptional.
\end{prop}
\begin{proof}
We have a triangulated equivalence:
$$
F \colon D^b(A) \to D^b(\Coh (\mathbb{X})).
$$
We show that $\Ext^1_{\Coh (\mathbb{X})} (FM,FM)=0$, we have that
$$
\Ext^1_{\Coh (\mathbb{X})} (FM,FM) = \Hom_{D^b(\Coh (\mathbb{X}))} (FM,FM[1]) \cong \Ext^1_{A} (M,M).
$$

By proposition \ref{bundle}, $FM$ in $\Coh (\mathbb{X})_+$. 
Therefore, $FM$ is a bundle. 
Then we use the theorem of Lenzing-Reiten \cite{len-rei} which states that if $\Coh (\mathbb{X})$ of domestic type, 
then each indecomposable bundle is exceptional.

Hence, the sincere indecomposable module $M$ is exceptional.
\end{proof}

We state now a particular case of a result of  Lenzing- De la Pe\~na, which we will use in \ref{phiawild}.
\begin{teo} [Lenzing-De la Pe\~na \cite{len-pen}] 
Let $A$ be an algebra derived equivalent to a canonical algebra $\Lambda$ of weight type $(p_1, \dotsc , p_n)$. 
We fix a triangle-equivalence $D^b (A) \cong D^b(\Coh (\mathbb{X}))$, where $\Coh (\mathbb{X}) \cong' \Lambda$. 
Let $M$ be an exceptional $A$-module. If 
the module $M$ corresponds to a shift of  an exceptional vector bundle over $\mathbb{X}$, and 
$\mathcal{X}(\mathbb{X}) > 0$, 
then $A[M]$ is derived equivalent to the path algebra of a wild connected quiver.

\end{teo}

\begin{lema} \label{phiawild}
Let $A$ be a representation-infinite quasi-tilted algebra of domestic-sheaf type with a sincere indecomposable module $M$. Then 
$A[M]$ is derived equivalent to the path algebra of a wild connected quiver.
\end{lema}
\begin{proof}
We have a triangulated equivalence:
$$
F \colon D^b(A) \to D^b(\Coh (\mathbb{X})).
$$
By proposition \ref{excep}, $M$ is exceptional. Moreover, by proposition \ref{bundle}, $M$ associated, by $F$, with a exceptional vector bundle 
over $\mathbb{X}$.

We use the theorem of Lenzing-De la Pe\~na \cite{len-pen} above. We conclude that $A[M]$ is derived equivalent to the path algebra of a wild connected quiver.
\end{proof}

We can produce several examples of PHI algebras where $K\Delta[M]$ is of wild type.

\begin{example}
We give now an example of a PHI algebra $K\Delta$ of type $\widetilde{\mathbb{E}}_6$. Using our lemma, \ref{phiawild}, we see that  the one point extension of $K\Delta$ by the 
canonical sincere module $M$  is a PHI algebra of wild type. 

\begin{align*}
K\Delta :
\begin{aligned}
\xymatrix@R8pt@C10pt{
\bullet \ar[r]	& \bullet \ar[dd]		&& \bullet \ar[ll] \ar[dr]	&\\
				&						&&							& \bullet \ar[dl]\\
\bullet 		& \bullet \ar[rr] \ar[l]&& \bullet \ar@{.}[uu]		&
}
\end{aligned}
&&
K\Delta[M] :
\begin{aligned}
\xymatrix@R8pt@C10pt{
\bullet \ar[dr]	&						&& \star \ar[d] \ar[lll] \ar@{.}[dll]	&\\
				& \bullet \ar[dd]		&& \bullet \ar[ll] \ar[dr]			&\\
				&						&&									& \bullet \ar[dl]\\
\bullet 		& \bullet \ar[rr] \ar[l]&& \bullet \ar@{.}[uu]				&
}
\end{aligned}
\end{align*}
\end{example}

%\section*{References}

% Bibliografia
\bibliographystyle{amsplain} % citação bibliográfica alpha
\bibliography{ref}  % associado ao arquivo: 'ref.bib'

\end{document}